\documentclass[12pt]{article}
\usepackage{amsmath, amssymb, amsthm}
\usepackage{enumerate}

\newtheorem*{propositiona}{Proposition}
\newtheorem*{theorema}{Theorem 1}
\newtheorem*{theoremb}{Theorem 2}
\newtheorem*{theoremc}{Theorem 3}
\newtheorem*{theoremd}{Change of Variable Formula}
\newtheorem*{theoreme}{Proof, Change of Variable Formula}
\newcommand{\R}{{\mathbb{R}}}
\newcommand{\taf}{{\hskip 5pt} $\blacksquare$
                  \renewcommand{\qedsymbol}{}}
\begin{document}
\title{The change of variable formula for the Riemann integral}
\author{Alberto Torchinsky}
\date{}
\maketitle

This note concerns  the  general formulation by Preiss and Uher \cite{Pries}
of  Kestelman's result pertaining the  change of variable, or substitution, formula for the Riemann integral  \cite{Davies}, \cite{Kestelman}.

Specifically, we prove
\begin{theoremd}  
 Let $\varphi$ be a bounded, Riemann integrable function defined on an interval $I=[a,b]$, and let $\Phi$ be an indefinite integral of $\varphi$ on $I$.
Let  $f$ be bounded on $\Phi(I)$, the range of $\Phi$. Then,  $f$ is   Riemann integrable on $\Phi(I)$   iff $ f(\Phi)\varphi $ is Riemann integrable on $I$, and, in that case, with $\mathcal I=[\Phi(a), \Phi(b)]$, 
\begin{equation} \int_{\mathcal I} f= \int_I  f(\Phi)\,\varphi \,.
\end{equation}
\end{theoremd}

Our approach  is   self-contained and  elementary, since only basic properties in the  theory of the Riemann integral  are   invoked. Developments in this area since Kestelman's influential   paper,
 as well as  the various strategies utilized, 
can be found in \cite{Bagby},  \cite{Lopez}, \cite{Navra}, \cite{Sarkel}, [14,15], \cite{Thomson1},  and the references therein.

We begin by introducing some definitions and notations. 
Fix a closed finite interval  $I=[a,b]\subset \R$, and let $\Phi$ be a continuous increasing function defined on $I$. For a partition $\mathcal P=\{I_k\}$ of $I$, where $I_k=[x_{k,l},x_{k,r}]$, and a bounded function $f$  on $I$, let  
 $U(f,\Phi, \mathcal{P})$ and  $L(f,\Phi, \mathcal{P})$ denote  the  upper and lower Riemann sums of $f$ with respect to $\Phi$ on $I$ along $\mathcal{P}$,  i.e.,
\[U(f,\Phi, \mathcal{P})=\sum_k \big(\sup_{I_k}f\big)\, \big(\Phi(x_{k,r})- \Phi(x_{k,l})\big),\]
and  
\[ L(f,\Phi, \mathcal{P})=\sum_k \big(\inf_{I_k}f\big)\, \big(\Phi(x_{k,r})- \Phi(x_{k,l})\big),
\]
respectively, and set
 \[ U(f,\Phi)=\inf_{\mathcal{P}}\, U(f,\Phi,\mathcal{P})\,,\quad{\rm{and}} \quad
 L(f,\Phi)=\sup_{\mathcal{P}} L(f,\Phi, \mathcal{P})\,.
 \]
 
We say that $f$ is Riemann integrable with respect to $\Phi$ on $I$ if $U(f,\Phi)=L(f,\Phi)$, and in this case the common value is denoted $\int_I f\,d\Phi$, the Riemann integral of $f$ with respect to $\Phi$ on $I$. 

When $\Phi(x)=x$ one gets the usual Riemann integral on $I$, and  $\Phi$ is omitted in the above notations. And, throughout this note, when it is clear from the context, integrable means Riemann integrable with respect to $\Phi(x)=x$, and Riemann-Stieltjes integrable means integrable with respect to a more general $\Phi$.
 
 We begin by proving a working characterization of integrability \cite {Bruckner}, \cite{Hunter},
 \cite{Torch},

\begin{propositiona} Let $f$ be a bounded function defined on $I$. Then, 
 $f$ is Riemann integrable with respect to $\Phi$  on $I$ iff,  given $\varepsilon >0$, there is a partition ${\mathcal{P}}$ of $I$, which 
may  depend on $\varepsilon$,  such that 
\begin{equation}  U(f,\Phi,{\mathcal{P}})-L(f,\Phi,{\mathcal {P}})\le \varepsilon\,.
\end{equation}

Furthermore, a sequential characterization holds, to wit,
 (2) is equivalent to the existence of a sequence $\{\mathcal P_n\}$ of partitions of $I$ such that
\begin{equation}
\lim_{n}  \big( U(f,\Phi,\mathcal { P}_n) - L(f,\Phi, \mathcal {P}_n) \big) = 0\,,
\end{equation}
and, in this case,
\begin{equation} \lim_{n} U(f,\Phi, \mathcal P_n) = \lim_{n}L(f, \Phi, {\mathcal P}_n)=
\int_I f\,d\Phi.
\end{equation}
\end{propositiona}

\begin{proof}
First, if $f$ is Riemann-Stieltjes integrable, given   $\varepsilon > 0$, there are partitions $\mathcal Q, \mathcal R$ of $I$, such that
\[U(f,\Phi,\mathcal Q) \le U(f,\Phi) +\varepsilon/2\,,\quad {{\rm and}}\quad L(f,\Phi) \le L(f,\Phi,\mathcal R) + \varepsilon/2\,,\]
and so, by the monotonicity properties of the upper and lower sums, for a  common refinement $\mathcal P$ of $\mathcal Q$ and $\mathcal R$, we have 
\[U(f,\Phi,\mathcal P) - L(f,\Phi,\mathcal  P) \le U(f,\Phi,\mathcal Q) - L(f,\Phi,\mathcal R) \le U(f,\Phi) - L(f,\Phi) + \varepsilon,\]
which, since $U(f,\Phi) = L(f,\Phi)$, gives (2).

Now, if (2) holds,  given  $\varepsilon > 0$, pick a partition
$\mathcal P$ of $I$ that satisfies (2). Then, since $U(f,\Phi)\le  U(f, \Phi,\mathcal P)$ and $L(f,\Phi, \mathcal P) \le L(f,\Phi)$,
we have
$0 \le U(f,\Phi) - L(f,\Phi) \le U(f,\Phi,\mathcal P) - L(f,\Phi,\mathcal P) <\varepsilon$,
which, since $\varepsilon > 0$ is arbitrary, implies that  $U(f,\Phi)= L(f,\Phi)$, and $f$ is Riemann-Stieltjes integrable.

As for the sequential characterization, if $f$ is Riemann-Stieltjes integrable,   for each integer $n$ pick a partition 
$\mathcal P_n$ of $I$ such that $0 \le  U(f,\Phi,\mathcal P_n) - L(f,\Phi, \mathcal P_n) \le 1/n\,$;
clearly  $\lim_n \big(U(f,\Phi,\mathcal P_n) -L(f,\Phi, \mathcal P_n)\big)= 0$.
Furthermore, since $U(f,\Phi) \le U(f,\Phi,\mathcal P_n)$ and $L(f,\Phi,\mathcal P_n) \le L(f,\Phi)$, it follows that
\[ U(f,\Phi, \mathcal P_n) - U(f,\Phi) = U(f,\Phi, \mathcal P_n) - L(f,\Phi)\le 
U(f,\Phi, \mathcal P_n) - L(f,\Phi, \mathcal P_n)\,,\]
and, consequently, $\lim_n U(f,\Phi,\mathcal P_n) = U(f,\Phi)$. Finally ,
\[\lim_n L(f,\Phi,\mathcal P_n) = \lim_n U(f,\Phi,\mathcal P_n) 
  - \lim_n \big(U(f,\Phi,\mathcal P_n) -
 L(f,\Phi,\mathcal P_n)\big) = \int_I f\,d\Phi,\]
 and (3) and (4) hold.

Conversely, clearly (3) implies (2), and so  $f$ is Riemann-Stieltjes  integrable.
\taf\end{proof}

Note that if (2) holds for a  partition ${\mathcal P}$, 
it also holds for  partitions  ${\mathcal{P}'}$ 
 finer than ${\mathcal{P}}$. This  observation also applies to other concepts introduced here, including (4), and by (36), also to (7)  and (8) below.

Now, since $\Phi$  is continous on $I$, $\Phi(I)=\mathcal I=[\Phi(a), \Phi(b)]$   is an interval   with endpoints $\Phi(a)$ and $ \Phi(b)$.  Note that each interval $\mathcal J=[y_1,y_2] \subset \mathcal I$ is of the form $[\Phi(x_1),\Phi(x_2)]$, where $\Phi(x_1)=y_1, \Phi(x_2)=y_2$, and $[x_1, x_2]$ is a subinterval of $I$. Moreover, partitions $\mathcal P$ of $I$  induce a corresponding partition $\mathcal Q$ of $\mathcal I$, and, conversely,  every partition of $\mathcal I$ can be expressed as $\mathcal Q$ for some partition $\mathcal P$ of $I$.

We prove next three basic results,  of independent interest,
on Riemann integration. 
The first result involves the notion of oscillation of a function.
Recall that, given a bounded function $g$ defined on $I$  and an interval $J\subset I$, the oscillation ${\rm {osc\, }}(g, J)$ of $g$ on $J$ is defined as 
$ {\rm {osc\, }}(g, J) = \sup_J g - \inf_J g$.

We then have \cite{Bruckner},  \cite{Hunter}, 
\begin{theorema} Let $g$ be a bounded function on $I$. Then, $g$ is Riemann integrable with respect to $\Phi$ on $I$ iff  given $\varepsilon >0$, there is a partition ${\mathcal{P}}=\{I_k\}$ of $I$, which 
may  depend on $\varepsilon$,  such that 
\begin{equation}  \sum_k {\rm {osc\, }}( g, I_k)\,\big(\Phi(x_{k,r})- \Phi(x_{k,l})\big)\le\varepsilon\,.
\end{equation}

Furthermore, a sequential characterization  holds, to wit,  $g$ is Riemann integrable with respect to $\Phi$ on $I$ iff there exists a sequence $\{\mathcal {P}_n\}$ of partitions of $I$ consisting of the intervals $\mathcal {P}_n=\{I_k^n\}$ such that 
\begin{equation*} \lim_n \sum_k {\rm {osc\, }}( g, I_k^n)\,\big(\Phi(x_{k,r}^n)- \Phi(x_{k,l}^n)\big)=0\,.
\end{equation*}
\end{theorema}

\begin{proof}
Note that for each partition $\mathcal P=\{I_k\}$ of $I$,
\begin{equation} U(g,\Phi, \mathcal {P})-  L(g,\Phi, \mathcal {P})= \sum_k {\rm {osc\, }}(g, I_k)\,\big(\Phi(x_{k,r})- \Phi(x_{k,l})\big)\,.
\end{equation}

Now, if $g$ is Riemann-Stieltjes integrable, given $\varepsilon>0$, by (2), 
pick a partition ${\mathcal{P}}$ of $I$ such that  $ U(f,\Phi,{\mathcal{P}})-L(f,\Phi,{\mathcal {P}})\le \varepsilon$, and observe that (6) implies (5).  Conversely,   given $\varepsilon>0$, pick a partition such that (5) holds, and observe that by (6) also (2) holds, and consequently, $g$ is Riemann-Stieltjes integrable. 

The proof of the sequential convergence is analogous once we invoke (3), rather than (2), above,  and is left to the reader.
\taf
\end{proof}

For the special case of the Riemann integral, Theorem 1 expresses quantitatively the  fact that Riemann integrable functions are continuous a.e.\! with respect to the Lebesgue measure. It states that,  $ g$ is Riemann integrable  on $I$ iff given $\varepsilon>0$, there is a partition $\{\mathcal {P}\}$ of  $I$ such that 
\begin{equation}
\sum_k {\rm {osc\, }}( g, I_k)\,|I_{k}|\le \varepsilon\,,
\end{equation}
 iff there exists a sequence $\{\mathcal {P}_n\}$ of partitions of $I$ consisting of the intervals  $\mathcal {P}_n=\{I_k^n\}$ such that  
\begin{equation}
 \lim_n \sum_k {\rm {osc\, }}( g, I_k^n)\,|I_{k}^n|=0\,.
\end{equation}

The next result relates a Riemann-Stieltjes integral to a Riemann integral.  This  is of some interest because for  integrable $f$, the composition $f (\Phi)$ turns out to be Riemann-Stieltjes integrable, although  $f (\Phi)$ may fail to be   integrable,  even if $\Phi$ is continuous \cite{Gelbaum}, \cite{Kestelman}.

\begin{theoremb} Let  $f$ be a bounded function on $\Phi(I)=\mathcal I$. Then,  $f$ is integrable on $\mathcal I$ iff $ f(\Phi)$ is Riemann  integrable with respect to $\Phi$ on $I$,
and, in that case  we have 
\begin{equation} \int_{\mathcal I} f=\int_I f(\Phi)\, d\Phi\,.
\end{equation}
\end{theoremb}

\begin{proof}
Specifically,  (9) is understood to mean that, if either side of the equality exists, so does the other side and they are equal. 
To see this, let the  partition  $\mathcal Q =\{\mathcal I_k\}$ of $\mathcal I$ correspond to the partition $\mathcal P=\{ I_k\}$ of $I$ such that $\mathcal I_k =[\Phi(x_{k,l}),\Phi(x_{k,r}) ]$, where $I_k=[x_{k,l}, x_{k,r}]$. Then, since
$  \sup_{{\mathcal I}_k} f = \sup_{I_k} f(\Phi)$, and, $|{\mathcal I}_k| = \Phi(x_{k,r}) -\Phi(x_{k,l})$,  it readily follows that
\begin{align*}
U(f, \mathcal Q) &= \sum_k \big(\sup_{{\mathcal I}_k} f\big)\, |{\mathcal I}_k|\\
&= \sum_k 
 \sup_{I_k} f(\Phi) \big(\Phi(x_{k,r}) -\Phi(x_{k,l})\big) = U(f(\Phi), \Phi, \mathcal P)\,,
\end{align*}  
and, similarly,
$ L(f, \mathcal Q)= L(f(\Phi), \Phi, \mathcal P)\,.$ (9) follows at once from these identities.
\taf
\end{proof}

 The third result reduces the computation of a Riemann-Stieltjes integral to that of a Riemann integral. 
 Let $\varphi$ be a bounded, Riemann integrable function defined on $I=[a,b]$, and let $\Phi$ be an indefinite integral of $\varphi$ on $I$, i.e.,
\begin{equation} \Phi(x)=\Phi(a)+ \int_{[a,x]} \varphi\,,\quad  x\in I\,.
\end{equation}
Such functions  have been characterized in \cite{Thomson}.

We then have  \cite{Lopez1}, 

\begin{theoremc}  Let $\Phi$ be as in (10) with $\varphi$ positive, and let $g$ be a bounded function on $I$. Then,   $ g$ is Riemann integrable with respect to $\Phi$ on $I$ iff $g\,\varphi$ is Riemann integrable on $I$, and  in that case  we have  
\begin{equation}
\int_I g \,d\Phi=\int_{I} g\,\varphi \,,
\end{equation} 
in the sense that, if the integral on either side of (11) exists, so does the integral on the other side and they are equal.
\end{theoremc}

\begin{proof}
Assume first that $ g$ is Riemann-Stieltjes integrable, and fix $\varepsilon >0$. 
Then, for a partition $\mathcal P=\{I_k\}$ of  $I$ with $I_k=[x_{k,l},x_{k,r}]$, pick $ \xi_k\in I_k$ such that
\begin{equation} U( g\,\varphi,\mathcal P) \le \sum_k  g(\xi_k)\, \varphi(\xi_k)|I_k|+\varepsilon\,.
\end{equation}

There are two types of summands in (12),  to wit, those where
$g  (\xi_k)>0$ and those where $ g  (\xi_k) <0$. In the former case, note that
\begin{align} g (\xi_k)\varphi(\xi_k)|I_k| &= g  (\xi_k)\big( \varphi(\xi_k) -\inf_{I_k}\varphi \big)|I_k| +g  (\xi_k)\big( \inf_{I_k}\varphi\big)\, |I_k|\nonumber \\
\le g  &(\xi_k )\, {\rm {osc\, }}(\varphi, I_k)\,|I_k| + g  (\xi_k)\,\int_{I_k}\varphi 
\nonumber\\
&= g  (\xi_k)\, {\rm {osc\, }}(\varphi, I_k)\,|I_k| + g  (\xi_k)\,\big( 
\Phi(x_{k,r})-\Phi(x_{k,l})\big)\,, 
\end{align}
and in the latter case, since $\int_{I_k}\varphi\le \big(\sup_{I_k}\varphi \big)\, |I_k|$, 
it   follows that
\begin{align} g  (\xi_k)\varphi(\xi_k)|I_k| &= 
- g  (\xi_k)\big(\sup_{I_k}\varphi- \varphi(\xi_k)\big)|I_k| + (- g  (\xi_k)) 
(-\big( \sup_{I_k}\,\varphi\big)\, |I_k |\,)\nonumber\\
\le | g &(\xi_k)|\, {\rm {osc\, }}(\varphi, I_k)\,|I_k| +(- g (\xi_k))\,
\big( -\int_{I_k}\varphi \big)\nonumber
\\
&=| g  (\xi_k)|\, {\rm {osc\, }}(\varphi, I_k)\,|I_k| + g  (\xi_k)\,\big( 
\Phi(x_{k,r})-\Phi(x_{k,l})\big)\,. 
\end{align}

Hence, adding (13) and (14), with $M_g$ a bound for $g$,  we have
\begin{align*} \sum_k  g  (\xi_k) \varphi(\xi_k)|I_k| &\le M_g \sum_k  {\rm {osc\, }}(\varphi, I_k)\,|I_k| +\sum_k 
 g  (\xi_k)\,\big( \Phi(x_{k,r})-\Phi(x_{k,l})\big)\\
&\le M_g \sum_k   {\rm {osc\, }}(\varphi, I_k)\,|I_k|  + U(g, \Phi,\mathcal P)\,,
\end{align*}
which, by (12), implies that 
\begin{equation}  U(g\,\varphi,\mathcal P)\le M_g  \sum_k  {\rm {osc\, }}(\varphi, I_k)\,|I_k|  + U(g, \Phi,\mathcal P) +\varepsilon.
\end{equation}

Applying (15) to $-g$ gives  
\begin{equation*} -L(g\,\varphi, \mathcal P) \le 
 M_g \sum_k {\rm {osc\, }}(\varphi, I_k)\,|I_k|  - L(g, \Phi,\mathcal P) +\varepsilon,
\end{equation*}
and adding to (15) we get 
\begin{align} U(g\,\varphi,\mathcal P) &-  L(g\,\varphi,\mathcal P)\nonumber\\
&\le 2\, M_g \sum_k  {\rm {osc\, }}(\varphi, I_k)\,|I_k| +\big(
 U(g, \Phi,\mathcal P) -  L(g, \Phi,\mathcal P)\big)  +2\,\varepsilon.
\end{align}

Let $\mathcal {P}$ be a  partition of $I$ that satisfies simultaneously 
(7) for $\varphi$ and (2) for $g$ with respect to $\Phi$ for the $\varepsilon>0$ picked at the beginning of the proof; a common refinement of a partition that satisfies  
(7) for $\varphi$ and one that satisfies (2) for $g$ with respect to $\Phi$ will do.   Then from (16) it readily follows that
$ U(g\,\varphi,\mathcal P) -  L(g\,\varphi,\mathcal P)\le 2\, M_g\,\varepsilon + \varepsilon +2\, \varepsilon,
$ and, therefore, since $\varepsilon>0$ is arbitrary, by (2),  $g\varphi$ is integrable on $I$.    

It only remains to evaluate the integral in question. Let $\{\mathcal {P}_n\}$ be a sequence of partitions of $I$ that satisfies simultaneously (8)  for $\varphi$ and (4) for $g$. Then, given $\varepsilon>0$,  from (15) it follows that
\begin{align*}\int_I g\,\varphi&= U(g\,\varphi)\le \limsup_n U(g\,\varphi, \mathcal P_n)\\
&\le \limsup_n  M_g \sum_k  {\rm {osc\, }}(\varphi, I_k^n)\,|I_k^n|+ \limsup_n  U( g\,,\Phi\,,\mathcal P_n)+\varepsilon,\\
& =\int_I g\,d\Phi +\varepsilon,
\end{align*}
which, since $\varepsilon$ is arbitrary, gives
$ \int_I g \,\varphi \le  \int_I g\,d\Phi\,.$ 
Furthermore,   replacing $g$ by $-g$  it follows that $ \int_I g\,d\Phi\le \int_I g\,\varphi \,,$ 
(11) holds, and the conclusion obtains.

The proof of the converse requires no new ideas and we will be brief. Assume that $g\varphi$ is integrable on $I$,  let   
 $\mathcal P =\{I_k\}$ be a partition of $I$, and, given $\varepsilon>0$, pick
$\xi_k\in I_k$ such that
\begin{equation*}
U (g,\Phi,\mathcal P)= \sum_k \big( \sup_{I_k } g\big)\, \big ( \Phi(x_{k,r}) -\Phi( x_{k,l})\big)
\le \sum_{k } g (\xi_k)  \int_{I_k} \varphi +\varepsilon.
\end{equation*}

Proceeding as in the first part of the proof we arrive at an analogous relation   to (16), but with the Riemann sums for $g$ and  the Riemann sums for $g$ with respect to $\Phi$, switched, to wit, 
\begin{align*} U(g, \Phi,\mathcal P) &-  L(g, \Phi,\mathcal P)\\
&\le 2\, M_g \sum_k  {\rm {osc\, }}(\varphi, I_k)\,|I_k| +\big(U(g\,\varphi,\mathcal P)  -  L(g\,\varphi,\mathcal P)\big) +2\,\varepsilon.
\end{align*}
As above  we conclude that $ g$ is Riemann-Stieltjes integrable and, therefore, invoking the first part of the proof,  $ \int_I g\,d\Phi$
$= \int_I g\,\varphi $,  (11) holds, and the proof is finished.
\taf
\end{proof}

\begin{theoreme}
\end{theoreme}
We are now ready to prove the  change of variable formula. It is at this juncture  that we drop the assumption that $\varphi$ is positive  and allow it to change signs; thus, the substitution is not required to be invertible. Then $\Phi(I)$, the range of $\Phi$, is an interval, but $\Phi(a), \Phi(b)$ are not necessarily endpoints of this interval. 
It is important to keep in mind  that the Riemann integral is oriented, and that the direction in which the interval is traversed determines the sign of the integral. 
Also note that the assumption that  $f$ is  bounded is   necessary,  as a simple example shows \cite{Sarkel}.  And, some care must be exercised  since  for $f(\Phi)\varphi$   integrable on $I$ and $\varphi$   continuous on $I$, it does not follow    that $f(\Phi)$ is    integrable on $I$, \cite{Kestelman}. 

The proof is carried out in two parts,  when $\varphi$ is of constant sign, and  when $\varphi$ is of variable sign. In the former case, suppose first that $\varphi$ is positive. Then, if $f$ is integrable on $\mathcal I$,   by Theorem 2, $ f(\Phi)$ is Riemann-Stieltjes integrable, and $\int_{\mathcal I}f=\int_I 
f(\Phi)\,d\Phi$. And, by Theorem 3 with $g=f(\Phi)$ there,   $ f(\Phi)\,\varphi$ is Riemann integrable on $I$, and 
$\int_I f(\Phi)\,d\Phi=\int_I f(\Phi)\,\varphi$. This chain of arguments shows that if $f$ is  integrable on $\mathcal I$,  $ f(\Phi)\,\varphi$ is  integrable on $I$,  and
$  \int_{\mathcal I}f= \int_I f(\Phi)\,\varphi\,.$   
Moreover, since all the steps in the above argument are reversible, the converse also holds, and the substitution formula has been established in this case. 

When $\varphi$ is negative, let $\psi(x)= -\varphi(x)$, and $\Psi(x)=-\Phi(x)$.  Then by (1) applied to $g(x)=f(-x)$, it follows that
\begin{equation} \int_{[\Psi(a), \Psi(b)]} g= \int_I  g(\Psi)\,\psi \,,
\end{equation}
where the left-hand side of (17) is equal to
$ \int_{[-\Phi(a), -\Phi(b)]} g =- \int_{[\Phi(a), \Phi(b)]} f
\,,$ 
and the right-hand side of (17)  equals
$ \int_I  g(\Psi)\,\psi = \int_I  g(- \Phi)\,(-\varphi) =
-  \int_I  f( \Phi)\,\varphi\,.$
Hence, the substitution formula holds when $\varphi$ is of constant sign, and the first part of the proof is finished.

Next, consider when $\varphi$ is of variable sign. First,  assume  that $f$ is  integrable on $\Phi( I)$. The idea is to show that $\int_{\Phi(I)}f$  can be approximated arbitrarily close by the Riemann sums of $f(\Phi)\varphi$ on $I$, and, consequently,  $\int_I f(\Phi)\varphi$  also exists, and the integrals are equal \cite{Bagby}, \cite{Pries}. To make this argument precise we  begin by introducing the partitions used for the approximating Riemann sums. They are based on a partition  $\mathcal P$ of $I$ defined as follows: given $\eta >0$, by (7)  there is a partition $\mathcal P=\{I_k\}$  of $I$, such that
\begin{equation} \sum_{k} {\rm {osc\, }}(\varphi,  I_k)\, |I_k|\le \eta^2 |I|\,.
\end{equation}

We  first separate the indices $ k$ that appear in $\mathcal P$ into  three classes, the (good) set $G$, the (bounded) set $B$, and the (undulating) set $U$,   according to the following criteria. First, $k\in G$ if  $\varphi$ is strictly positive or negative on $I_k$. Next, $k\in B$, if $k\notin G$ and $|\varphi|\le \eta$ on $I_k$. And, finally, $k\in U$, if $k\notin G\cup B$.  Note that for $k\in U$, since $\varphi$ changes signs in $I_k$ and for at least one point $\xi_k$ there, $| \varphi(\xi_k)|> \eta$, we have  $ {\rm {osc\, }}(\varphi,  I_k)\ge \eta$. 

Recall that each $I_k=[x_{k,l},x_{k,r}]$ in $\mathcal P$   corresponds to the (oriented) subinterval  $\mathcal I_k=[\Phi(x_{k,l}),\Phi(x_{k,r})]$
of $\Phi(I)$. Now, since  $f$ is  integrable on $\Phi( I)$, $f$ is  integrable on $\mathcal I_k$, and if $k\in G$,  by the first part of the proof, $ f(\Phi)\varphi $ 
is integrable on $I_k$, and 
$\int_{\mathcal I_{k}} f = \int_{I_k}f(\Phi)\,\varphi$.
Then, by (2), given $\eta>0$, there is a partition  $\mathcal P^k= \{I_j^k\}$ of $I_k$ such that 
\begin{equation*} U(f(\Phi)\varphi,\mathcal P_k)-L(f(\Phi)\varphi,\mathcal P_k)=
\sum_j {{\rm osc}}\,( f(\Phi)\varphi,I_j^k )\, | I_j^k| \le \eta \,|I_k|\,.
\end{equation*} 
Moreover, since $  \int_{\mathcal I_k} f\le  U(f(\Phi)\varphi, \mathcal P^k)$, 
we also have
\begin{equation*} 
 U(f(\Phi)\varphi, \mathcal P^k) -\int_{\mathcal I_k} f \le\eta\,|I_k|\,.
\end{equation*}

Hence, for $k\in G$,
\begin{equation} 
{{\sum_{k\in G}\sum_j \rm osc}}\,(f(\Phi)\varphi, I^k_j)\, |  I^k_j|\le \eta \sum_{k\in G} |I_k|,
\end{equation}
and
\begin{equation}
\sum_{k\in G}\Big| \int_{\mathcal I^k} f  - U(f(\Phi)\varphi,\mathcal P^k)\Big| \le  \eta \sum_{k\in G} |I_k|. 
\end{equation}

 Now, for $k\in B\cup U$, let $\mathcal P^k=\{I_k\}$ denote the partition of $I_k$ consisting of the interval $I_k$. Note that, with $M_{\varphi}$  a bound for $\varphi$,  
\begin{equation}|\mathcal I_{k}|=|\Phi(x_{k,r})-\Phi(x_{k,l})|\le \int_{[x_{k,l},x_{k,r}]}|\varphi| \le M_{\varphi}\,|I_k|\,,
\end{equation}
and,  with $M_f$ a bound for $f$, that
\begin{equation}
\Big| \int_{\mathcal I_{k}} f\Big|\le M_f M_{\varphi}\,|I_k|\,.
\end{equation}

First, observe that
\begin{equation}
{{\rm osc }}\, (f(\Phi)\varphi,  I_k)\,| I_k| =\big(\sup_{ I_k} f(\Phi)\varphi  - \inf_{ I_k} f(\Phi)\varphi\big) | I_k|  \le 
2 M_fM_\varphi\,|I_k|\,.
\end{equation}

Next, by (21) and (22), for  $\xi_k\in I_k$,
\begin{equation} \Big| \int_{\mathcal I_{k}} f - f(\Phi(\xi_k))\varphi(\xi_k)\,|I_k|
\Big|\le  2\,  M_f\, M_{\varphi}\, |I_k|\,,
\end{equation}
and so,  picking  $\xi_k\in I_k$ appropriately,  we have
\begin{equation}\Big| \int_{\mathcal I_{k}} f - U( f(\Phi)\varphi\,, \mathcal P^k)\Big|
\le 3\,  M_f\, M_{\varphi}\, |I_k|\,.
\end{equation}

Now, if $k\in B$,  $M_\varphi\le\eta$, and, therefore, from (23)  we get  that
\begin{equation}
\sum_{k\in B}{{ \rm osc}}\,(f(\Phi)\varphi, I^k)\, |  I^k_j|\le   2 M_f\, \eta\sum_{k\in B} |I_k|,
\end{equation}
and by (25),
\begin{equation}
\sum_{k\in B}\Big| \int_{\mathcal I_{k}} f - U(f(\Phi)\varphi, \mathcal P^k)\Big|\le 3\,  M_f\, \eta\, \sum_{k\in B}|I_k| 
\end{equation}
 
Finally, since for $k\in U$ we have osc\,($\varphi, I_k)\ge\eta$, as in Chebychev’s inequality,  from (18)    it follows that
\[\eta\sum_{k\in U}|I_k| \le\sum_{k\in U} {\rm {osc\, }}(\varphi,  I_k)\, |I_k|\le \sum_{k} {\rm {osc\, }}(\varphi,  I_k)\, |I_k|\le \eta^2 |I|\,,\]
and, consequently,
\begin{equation} \sum_{k\in U} |I_k|\le  \eta\,|I|\,.
\end{equation} 

Whence, by (23) and (28), the $U$ terms are bounded by 
\begin{equation}
\sum_{k\in U}{{ \rm osc}}\,(f(\Phi)\varphi, I^k)\, |  I^k|\le   2\, M_f\, M_\varphi \sum_{k\in U} |I_k|
\le   2\, M_f\, M_\varphi \eta\, |I|,
\end{equation}
and by (25) and (28), 
\begin{equation}\sum_{k\in U} \Big| \int_{\mathcal I_{k}} f -    U( f(\Phi)\varphi\,, \mathcal P^k) \,  \Big|\le  3\, M_f M_{\varphi}\,\sum_{k\in U}  |I_k| 
  \le 3\, M_f\,M_{\varphi}\,\eta\, |I|\,.
\end{equation}

Consider now the partition $\mathcal P'$ of $I$ that consists of the union of all the partitions $\mathcal P^k$,  where each $\mathcal P^k$ 
is defined  according as to whether $k \in G, k \in B$, or $k\in U$. Then, by (19), (26), and (29),
\begin{align}\sum_{k\in G}\sum_j {{\rm osc}}\, (f(\Phi)\varphi,\mathcal I_j^k )\,&|\mathcal I_j^k|  \nonumber\\
+\sum_{k\in B}
{{\rm  osc}} \, (f(\Phi)\varphi&, \mathcal I^k)\,|\mathcal I^k|+ \sum_{k\in U} {{\rm osc}}\, (f(\Phi)\varphi, \mathcal I^k) |\mathcal I^k|\nonumber
\\
\le \eta\sum_{k\in G}|I_k|+& 2\, M_f \,\eta \sum_{k\in B}|I_k| + 2\, M_f M_{\psi}\,\eta\, |I|\nonumber
\\
\le 
 \big( 1& + 2\, M_f  + 2\, M_f  M_\varphi\big)\,\eta\,|I|. 
\end{align}

Given $\varepsilon>0$, pick $\eta > 0$ so that
$( 1 + 2\, M_f+ 2\, M_f  M_{\varphi})\, \eta\,|I|\le \varepsilon$, 
and note that the above expression is $<\varepsilon$, and 
 since $\varepsilon>0$ is arbitrary,  (4) corresponding to $\mathcal P'$  implies that $f(\Phi)\varphi$ is Riemann  integrable, and $L(f(\Phi)\varphi,\Phi)= U(f(\Phi)\varphi,\Phi) =\int_I f(\Phi)\,\varphi$.

It remains to compute the integral in question. First, note that
\begin{equation}U(f(\Phi)\varphi, \mathcal P') =\sum_k U(f(\Phi)\varphi, \mathcal P^k)\,.
\end{equation}
Moreover,  since
$\Phi(b)-\Phi(a)=\sum_k \big(\Phi(x_{k,r})-\Phi(x_{k,l})\big)$,
by the linearity of the integral,
taking  orientation into account,   
 it follows that $\int_{\mathcal I}f =\sum_{k} \int_{\mathcal I_{k}} f$, \cite{Robbins}, \cite{Thomson1}. Hence, regrouping according to the sets $G, B$ and $U$, gives 
\begin{equation}
\int_{\mathcal I} f =\sum_{k\in G} \int_{\mathcal I_{k}} f + \sum_{k\in B} \int_{\mathcal I_{k}} f +\sum_{k\in U} \int_{\mathcal I_{k}} f \,, 
\end{equation}
and, from (32) and (33), it follows that
\begin{align*} \Big|\int_{\mathcal I} f -  U( f(\Phi)\varphi\,, \mathcal P')\Big| &\le  
 \sum_{k\in G}\Big| \int_{\mathcal I_k} f -  U( f(\Phi)\varphi\,, \mathcal P^k)\Big|
\\
+  \sum_{k\in B}\Big| \int_{\mathcal I_k} f &-  U ( f(\Phi)\varphi\,, \mathcal P^k)\Big|
+  \sum_{k\in U}\Big| \int_{\mathcal I_k} f - U( f(\Phi)\varphi\,, \mathcal P^k)\Big|\\
&= s_1 +s_2+ s_3\,,
\end{align*}
say.   
Now, by (20), 
\[ s_1\le  \sum_{k\in G} \Big| \int_{\mathcal I_k}  f -  U( f(\Phi)\varphi\,, \mathcal P^k) \Big| 
\le  \eta\, \sum_{k\in G} | I_k|\le \eta\,|I| \,,\]
and by (27) and (30),
$ s_2 +s_3 \le \big(  3 M_f +  3 M_fM_{\varphi} \big)\,\eta\,|I|\,,$
which combined give \[\Big| \int_{\mathcal I} f  - U(f(\Phi))\varphi, \mathcal P')\Big| \le    ( 1 + 3\, M_f + 3\, M_fM_{\varphi})\,\eta\,|I|\,.\]  

Given $\varepsilon>0$, pick $\eta > 0$ so that
$( 1 + 3\, M_f+ 3\, M_f M_{\varphi})\, \eta\,|I|\le \varepsilon$, 
 note that this $\varepsilon$ also works for (31), and that 
\begin{equation}\Big| \int_{\mathcal I}f- U( f(\Phi)\varphi\,, \mathcal P')\Big|\le  \varepsilon.
\end{equation}
Also, since $ U( f(\Phi)\varphi\,, \mathcal P')- L( f(\Phi)\varphi\,, \mathcal P')$ is equal to the left-hand side of (31), from (34) it follows that
\begin{equation}\Big| \int_{\mathcal I}f- L( f(\Phi)\varphi\,, \mathcal P')\Big|\le 2\, \varepsilon.
\end{equation}

Furthermore, since by (34),
\[\int_I f(\Phi)\varphi = U(f(\Phi)\varphi) \le U(f(\Phi)\varphi, \mathcal P') \le
\int_{\mathcal I} f + \varepsilon,\] 
and by (35), 
\[\int_{\mathcal I} f \le  L(f(\Phi)\varphi, \mathcal P')+ 2\,\varepsilon \le 
 L(f(\Phi)\varphi)+ 2\,\varepsilon = \int_I f(\Phi)\varphi+2\,\varepsilon,\]   we conclude that 
 \[\Big|\int_{\mathcal I} f-\int_I f(\Phi)\varphi\, \Big|\le 2\,\varepsilon,
 \] which, since $ \varepsilon$  is arbitrary, implies that
$\int_I f(\Phi)\varphi =\int_{\mathcal I} f\,.$ Hence, (1) holds, and the proof of this implication is finished.

As for the converse, 
it suffices to prove that if $ f(\Phi)\varphi$ is integrable on $I$,   $f$ is integrable on $\Phi(I)$, and  then invoke   the  implication we just proved. 

Let $\mathcal P$ be a partition  of $I$ that satisfies (18). Since $\Phi$ is continuous, $\Phi(I)$ is a closed interval of the form $ [ \Phi(x_m),\Phi(x_M) ]$  with (possibly non-unique) $x_m,x_M$ in $I$. If  $x_m$ and $x_M$  are endpoints of (not necessarily the same) interval in $\mathcal P$, proceed. Otherwise, since 
for  an interval $ J=[x_l,x_r]$ and  an interior  point $x$ of $  J$, with $  J_l=[x_l,x]$ and $  J_r=[x,x_r]$  we have 
\begin{equation}{{\rm osc}}\,(\varphi,  J_l)\,|  J_l|+ {{\rm osc}}\, (\varphi,  J_r)\,|  J_r|\le  {\rm osc}\, (\varphi,  J)\,|  J|\,,
\end{equation}
 $\mathcal P$ can be refined so that  the endpoint that was not  originally  included is now an endpoint
of two intervals of the new partition, without increasing the right-hand side
of (18).  For simplicity  also denote this new partition $\mathcal P$,   note that it contains both $x_m$ and $x_M$ at least once as an endpoint of one of its intervals, and define  the sets of indices   $G,B$, and $U$  associated to $\mathcal P$, as above. 

Now, if $f(\Phi)\varphi$ is  integrable on $I$, $ f(\Phi)\varphi$ is integrable on $I_k$, and, if $k\in G$,  by the first part of the proof, $ f $ 
is integrable on $\mathcal I_k$ and
$\int_{I_k}f(\Phi)\,\varphi =\int_{\mathcal I_{k}} f$. 
Then, by (7), given $\eta>0$, there is a partition  $\mathcal Q^k= \{{\mathcal I}_j^k\}$ of $\mathcal I_k$, such that 
\begin{equation*} \sum_j {{\rm osc}}\, (f, \mathcal I_j^k)\,|\mathcal I_j^k|\le \eta \,|I_k|\,,
\end{equation*}
and, therefore,
\begin{equation} \sum_{k\in G}\sum_j {{\rm osc}}\, (f, \mathcal I_j^k)\,|\mathcal I_j^k|\le \eta \,\sum_{k\in G}|I_k|\,.
\end{equation}

As for $k\in B\cup U$,  by (21) we get 
\[{{\rm osc }}\, (f, \mathcal I_k)\,|\mathcal I_k| =\big(\,\sup_{\mathcal I_k} f  - \inf_{\mathcal I_k} f \,\big) |\mathcal I_k|  \le 
2 M_fM_\varphi\,|I_k|\,.
\]

Next, if $k\in B$,   $M_\varphi\le\eta$, and, therefore, 
\begin{equation} \sum_{k\in B} {{\rm osc}}\, (f, \mathcal I_k)\,|\mathcal I_k|\le  2 M_f \,\eta \sum_{k\in B}|I_k|\,.
\end{equation}

Finally, for $k\in U$, by (28), $\sum_{k\in U} |I_k|\le \eta\,|I|\,,$ and so,
\begin{equation}\sum_{k\in U} {{\rm osc}}\,( f,\mathcal I_k) |\mathcal I_k| \le 2\, 
 M_f\, M_{\varphi}\sum_{k\in U} |I_k|\le  2\, M_f\,M_{\varphi}\,\eta\, |I|\,.
\end{equation}

Let $\mathcal Q'$ denote the collection   of subintervals of $\Phi(I )$  defined by $\mathcal Q'= \big( \bigcup_{k\in G} \bigcup_j  \{\mathcal I_j^k\} \big) \cup \big( \bigcup_{k\in B\cup U}\{\mathcal I_k\}\big).$ 
Note that the union of the intervals in $\mathcal Q'$ is $\Phi(I)$  and that, by (37), (38), and (39),  \begin{align}\sum_{k\in G}\sum_j {{\rm osc}}\, (f,\mathcal I_j^k &)\, |\mathcal I_j^k|  +\sum_{k\in B}
{{\rm  osc}} \, (f, \mathcal I^k)\,|\mathcal I^k|+ \sum_{k\in U} {{\rm osc}}\, (f, \mathcal I^k) |\mathcal I^k|\nonumber
\\
\le \eta\sum_{k\in G}&|I_k|+ 2 M_f \eta \sum_{k\in B}|I_k| + 2 M_f M_\varphi \eta\,|I|\nonumber
\\
&\le 
 \big( 1 + 2 M_f + 2 M_fM_\varphi\big)\,\eta\,|I|\,. 
\end{align}

Consider now the finite set  $\Phi(x_m)=y_1< y_2<  \cdots  <\Phi(x_M)=y_l$,  of   the endpoints of the intervals in $\mathcal Q'$  arranged in an increasing fashion, without repetition. Suppose that the interval $\mathcal J$ in $\mathcal Q'$ contains the points $y_{k_1}, \ldots, y_{k_n}$, say,  as endpoints or interior points. If they are endpoints, disregard them. Otherwise, as in (36),  incorporate each, from left to right,  as an  endpoint of two intervals in a refined  $\mathcal Q'$ without  increasing the right-hand side of (40). Clearly $\mathcal Q'$ thus refined contains a partition $\mathcal Q'' =\{\mathcal J_k\}$ of $\Phi(I)$, which, by    (40), satisfies,
\[ \sum_k {\rm osc}\,(f,\mathcal J_k)\,|\mathcal J_k|\le \big( 1 + 2 M_f + 2 M_fM_\varphi\big)\, \eta\,|I|. 
\]

Given $\varepsilon>0$, pick $\eta>0$ such that  $\big( 1+ 2 M_f +2 M_f M_{\varphi}\big)\, \eta\,|I|\le \varepsilon$.
 Then  the sum in (7) corresponding to $\mathcal Q''$ does not exceed an arbitrary $\varepsilon>0$, and, therefore, $f$ is integrable on $\Phi(I)$, and the proof is finished. \taf

A  caveat to the reader: not always the most general result is the most useful. By strengthening some assumptions and  weakening others in the Change of Variable Formula, it is possible to obtain a  substitution formula that does not follow from this result \cite{deO}. The same is true for Theorem 3.

Assume  that the function $\Phi$ is continuous,  non-decreasing on $I=[a,b]$, and differentiable on $(a,b)$  with derivative $\varphi\ge 0$;  then $\Phi$ is  uniformly continuous on $I$, and  maps $I$ onto $\mathcal I=[\Phi(a),\Phi( b)]$. We will also assume that $f$ is  Riemann integrable, rather than bounded, on $ \mathcal I$. On the other hand, we will not assume that (10) holds,  nor that $\varphi$ is bounded.
Then,  if   $f(\Phi)\varphi$ is integrable on $I$, the change of variable  formula holds.



To see this, consider a partition $\mathcal P=\{I_k\}$, $I_k=[x_{k,l},x_{k,r}]$, of $I$, and  the corresponding partition $\mathcal Q=\{\mathcal I_k\}$ of $\mathcal I$.
By the mean value theorem  there exist $\xi_k'\in I_k$ such that 
\begin{equation*}
y_{k,r}-y_{k,l}= \Phi(x_{k,r})- \Phi(x_{k,l}) = \varphi(\xi_k')\, \big(x_{k,r}-x_{k,l}\big)\,,\quad {{\rm all }}\ k\,,
\end{equation*}
and, therefore, with $\Phi(\xi_k' )=\zeta_k'\in\mathcal I_k$, 
\[ \sum_k   f(\Phi(\xi_k' ))\,  \varphi(\xi_k' )\,|I_k|= \sum_k f(\zeta_k')\,\big( y_{k,r} -  y_{k,l} \big),
\]
where the left-hand side is a Riemann sum of $f(\Phi)\,  \varphi$ on $I$, and the right-hand side a Riemann sum of $f$ on $\mathcal I$.
Since by the  uniform continuity of $\Phi$ it follows that $\max_k |I_k|\to 0$  
implies $\max_k|\mathcal I_k|\to 0$,  by the integrability assumptions, for appropriate partitions $\mathcal P$ the left-handside above tends to
$\int_I f(\Phi)\varphi$, and the right-hand side to $\int_{\mathcal I}f$. Hence the change of variable formula holds. 

This observation applies in the  following setting. On $I=\mathcal I=[0,1]$, with $0<\varepsilon<1$,  let $ \Phi(x)= x^{1-\varepsilon}$, and $\varphi(x)= (1-\varepsilon)\, x^{-\varepsilon}$ for $x\in (0,1]$, which is unbounded.
Then, for an integrable function $f$ on $\mathcal I$, provided that $f(\Phi)\,\varphi$ is integrable on $I$, the substitution formula holds. For $f$ we may take a  continuous function of order $x^\beta$ near the origin,  where $\beta\ge \varepsilon/(1-\varepsilon)$.

\end{document}